\documentclass[12pt,reqno]{amsart}
\usepackage{graphicx} 
\usepackage{enumerate,fullpage,xcolor}
\usepackage{amsmath, amsthm,amsfonts, amssymb, amscd} 
\usepackage{mathabx, latexsym, bbold}
\usepackage{tikz}
\usetikzlibrary{cd}

\newtheorem{theorem}{Theorem}[section]
\newtheorem{lemma}[theorem]{Lemma}
\newtheorem{proposition}[theorem]{Proposition}
\newtheorem{corollary}[theorem]{Corollary}

\newtheorem{rmk}[theorem]{Remark}
\theoremstyle{definition}
\newtheorem{definition}{Definition}[section]

\newcommand{\tab}{\hspace{3ex}}

\newcommand{\coker}{\text{coker\,}}
\newcommand{\im}{\text{im\,}}

\def\bC{\mathbb{C}}
\def\bZ{\mathbb{Z}}

\DeclareMathOperator{\Hom}{Hom}
\DeclareMathOperator{\Rep}{Rep}
\DeclareMathOperator{\End}{End}

\DeclareMathOperator{\irRC}{\operatorname{Irr}(\mathcal C)}

\newcommand{\mcA}{\mathcal{A}}
\newcommand{\mcB}{\mathcal{B}}
\newcommand{\mcC}{\mathcal{C}}
\newcommand{\mcD}{\mathcal{D}}

\newcommand{\mcI}{\mathcal{I}}

\newcommand{\mcK}{\mathcal{K}}

\newcommand{\mcN}{\mathcal{N}}

\newcommand{\mcP}{\mathcal{P}}

\newcommand{\mcS}{\mathcal{S}}
\newcommand{\mcT}{\mathcal{T}}

\newcommand{\mbbC}{\mathbb{C}}

\newcommand{\mbbZ}{\mathbb{Z}}

\title{Towards reconstruction of finite tensor categories}

\date{}

\author[ucsb]{Mitchell Jubeir\textsuperscript{(1)}}
\email{mitchell$\_$jubeir@ucsb.edu}

\author[ucsb]{Zhenghan Wang\textsuperscript{(1)}}
\email{zhenghwa@math.ucsb.edu}

\address[1]{Department of Mathematics, University of California, Santa Barbara, CA 93106, USA}

\begin{document}
\begin{abstract}
We take a first step towards a reconstruction of finite tensor categories using finitely many $F$-matrices.  The goal is to reconstruct a finite tensor category from its projective ideal.  Here we set up the framework for an important concrete example---the $8$-dimensional Nicholas Hopf algebra $K_2$.  Of particular importance is to determine its Green ring and tensor ideals. The Hopf algebra $K_2$ allows the recovery of $(2+1)$-dimensional Seiberg-Witten TQFT from Hennings TQFT based on $K_2$.  This powerful result convinced us that it is interesting to study the Green ring of $K_2$ and its tensor ideals in more detail. Our results clearly illustrate the difficulties arisen from the proliferation of non-projective reducible indecomposable objects in finite tensor categories.  

\end{abstract}
\maketitle
            
\section{Introduction}

Fusion categories and modular tensor categories are two well-studied classes of algebraic structures with wide interests from mathematics, physics, and quantum computing (see e.g. \cite{rowell2018mathematics, wang2010topological, nayak2008non}).  Non-semisimple generalizations to finite tensor categories and non-semisimple modular categories are expected to be interesting with connections to physics and quantum computing as well.  Invariants such as the modular $S$ and $T$ matrices, and finite presentations of fusion categories and modular tensor categories using $F$-matrices and $R$-matrices are crucial to their classification and applications (see e.g. \cite{rowell2009classification, ng2023reconstruction,levin2005string}).  We are interested in a generalization of modular data to non-semisimple modular categories and a reconstruction of finite tensor categories with finitely many $F$-matrices or $6j$-symbols.  Modular data of non-semisimple modular categories is recently investigated in \cite{chang2024modular}.  In this paper, we start to investigate how to present finite tensor categories by skeletalizing sub-categories which are complete invariants.  

A salient feature of non-semisimple Hopf algebras is the proliferation of reducible indecomposable modules that are not projective.  As we will see in Sec. \ref{sec: K2} that for the $8$-dimensional Nicholas Hopf algebra $K_2$, its indecomposable modules (up to isomorphisms) include a continuous family parameterized by the $2$-sphere.  Therefore, it is not even clear whether or not that finite tensor categories can be presented by a finite set of matrices as in the case of fusion categories.  We prove that the full subcategory of projectives is a complete invariant of a finite tensor category.  Hence to present a finite tensor category, it suffices to do so for its projective subcategory.  This in principle reduces the presentation, invariants, classification, and open problems of finite tensor categories to these of their projective subcategories. Note that the projective subcategory of a fusion category is itself. 

We will show that indeed finite tensor categories can be presented with finitely many matrices in a future publication. The analogue of $F$-matrices and $F$-symbols for a finite tensor category will be the analogous quantities for the indecomposable projecitve objects as a non-unital tensor category.  It remains to skeletalize such non-unital finite tensor categories by parameterizing pentagons with $F$-matrices. The projective fusion rules in general have very high multiplicities, therefore it would be in general much more challenging to solve the pentagons.  We will leave the explicit skelelatization, reconstruction, and its applications to a future publication.  In this paper, we content ourselves with the foundations and preparation for the example $K_2$.  

A more natural choice seems to be the quasi-dominated subcategory of a finite tensor category \cite{turaev2010quantum}, which includes both simple and projective objects.  The inclusion of simple objects makes an explicit reconstruction more transparent.  But it turns out that this subcategory is still unwieldy: in the case of $r(K_2)$, its quasi-dominated subcategory contains uncountably many indecomposable modules.

One preparation work for the example $K_2$ is a detailed study of its Green ring and tensor ideals.  The $8$-dimensional Nicholas Hopf algebra $K_2$ allows the recovery of $(2+1)$-dimensional Seiberg-Witten TQFT from Hennings TQFT based on $K_2$ \cite{kerler2003homology}.  This powerful result convinced us to study the Green ring of $K_2$ in more detail. One major goal is to determine the tensor ideals of the Green ring $r(K_2)$.  These results clearly illustrate the difficulties arisen from the proliferation of non-projective reducible indecomposable objects in finite tensor categories. In the future, we plan to classify all finite tensor categories with the same projective fusion rules.

In Sec. \ref{sec: nonunital}, we recall the Nicholas Hopf algebras $K_m$ and determine their projective Auslander algebras.  In Sec. \ref{sec: sub}, we prove that both the projective and quasi-dominated subcategories are complete invariants of finite tensor categories.  We also describe the categorical reconstruction of finite tensor categories from their projective subcategories. In Sec. \ref{sec: nicholas}, we show how to recover the Green ring of $K_{2m}$ if the Green ring of the double $DK_m$ is known. In Sec. \ref{sec: K2}, we determine the Green ring of $K_2$ and its tensor ideals.

\subsection{Notation and terminology}

We work over the field of complex numbers $\mathbb{C}$, so all vector spaces and algebras are finitely dimensional over $\mathbb{C}$.  Many results hold for general algebraically closed fields of characteristic zero.

\begin{enumerate}
    \item $\mcP_\mcC$=the projective subcategory of a finite tensor category $\mcC$.
    \item $A_\mcC$=the projective Auslander algebra of a finite tensor category $\mcC$
    \item $\mcN_\mcC$=the set of negligible objects of a  a finite tensor category $\mcC$
    \item $\mcS_\mcC$=the set of simple objects of a finite tensor category $\mcC$
    \item $\overline{S}^{\textrm{AC}}$=the additive subcategory whose objects are finite direct sums of elements of $S$ where $S$ is a subset of the FTC $\mcC$ known as the finite additive completion of $S$
    \item $\overline{\mcP_\mcC}^{\textrm{Ab}}$=the Abelian completion of the projective subcategory $\mcP_\mcC$

\end{enumerate}

\section{Non-unital tensor categories}\label{sec: nonunital}

A finite tensor category (FTC) $\mcC$ is a generalization of fusion category to the non-semisimple setting and the categorical counterpart of a finite-dimensional weak quasi-Hopf algebra.  We strictly follow the definition in \cite{EGNO,etingof2004finite}. In particular, an FTC is rigid, and every simple object $X$ has a unique projective cover $P(X)$ up to isomorphism.  Moreover, each projective indecomposable object covers a simple object called its socle. 
An FTC is unimodular if the projective cover $P_1$ of the tensor unit is self-dual.

Given an FTC $\mathcal{C}$ or the representation category $\Rep(H)$ of a Hopf algebra $H$, the full subcategory of projective indecomposable objects/modules lacks a monoidal unit of the tensor product in general.  In the semi-simple case, every simple object is also projective.    Therefore, it is interesting to study tensor categories without a tensor unit. We will call such a tensor category in the sense of \cite{etingof2004finite} without a tensor unit {\it a non-unital tensor category}.

Projective objects of an FTC $\mcC$ actually form a tensor ideal $\mcP_\mcC$ of $\mcC$.  Every FTC $\mcC$ has two trivial tensor ideals: the projetive $\mcP_\mcC$ and itself $\mcC$.  It is interesting and difficult to understand the complete set of intermediate tensor ideals between $\mcP_\mcC$ and $\mcC$ and their modified traces.

\subsection{Fusion rings, tensor ideals, and projective Auslander algebras}

\subsubsection{Green rings $r(\mcC)$ and projective fusion rings $K_0(\mathcal C)$}
Given an FTC $\mcC$, there are several different fusion rings and modules of $\mcC$ (see e.g. \cite{chang2024modular}).  The most complete one is the Green ring $r(\mcC)$.  In the case of a Hopf algebra $H$, the Green ring is simply the representation ring $\Rep(H)$ of $H$.  If $H$ is non-semisimple, then there are potentially uncountably many indecomposable modules of $H$.  

For an FTC $\mcC$, the Green ring $r(\mcC)$ is the fusion ring of all its indecomposable objects: simple objects, projective indecomposable objects, and indecomposable objects.  Equivalently, it can be defined as the representation ring of a weak quasi-Hopf algebra $H$ such that $\mcC=\Rep(H)$.  Therefore, the complexity for representation rings of non-semisimple Hopf algebras directly translates into that of FTCs. 

Another important fusion ring for us in this paper is the projective fusion ring.  Our goal is to encode all the information of an FTC by its projective subcategory $\mcP_\mcC$, and present $\mcP_\mcC$ by a finite set of matrices.

An important invariant of an FTC $\mcC$ is its projective Auslander algebra that we will define in this section.  The finitely dimensional projective Auslander algebra is a sub-algebra of the Auslander algebra, which is infinitely dimensional in general.

There is a related approach to the reconstruction of finite tensor categories \cite{chen2020reconstruction} under some finite conditions, which uses the Auslander algebra and the pentagons of the Green ring.  Instead, we will focus on the more manageable finite projective subcategories.   

The Green ring of a Hopf algebra is much more complicated than its Grothendieck ring because non-split exact sequence of representations leads to non-trivial relations.  The same applies to finite tensor categories.  The computation of the Green ring of a Hopf algebra is very challenging.  We mention a few known cases that are relevant to us here.

The Green rings of the Taft algebras are determined in \cite{chen2014greenTafts}.  Since the Sweedler Hopf algebra $K_1$ is a Taft algebra, it follows that the Green ring $r(K_1)$ of the Sweeder Hopf algebra is $r(K_1)=\mathbb{Z}[x,y]/I$, where $I$ is the ideal generated by $x^2-1$ and $y^2-y-xy$ \cite{chen2014greenTafts}.  The Green ring of the Drinfeld double of the Sweedler Hopf algebra is obtained in \cite{chen2014green}.  The Green rings of $K_m$ for $m\geq 2$ is unknown.  One of our results is the determination of the Green ring $r(K_2)$ of $K_2$.  We deduce this result from a quotient map from $DK_1$ to $K_2$ and the Green ring of the Drinfeld double of $K_1$ \cite{chen2014green}.

The projective sub-ring $K_0(\mathcal{C})$ of the Green ring is very important for this paper, so we recall its definition here.

\begin{definition} Let $K_0(\mathcal C)$ be the abelian group generated by the isomorphism classes $[P]$ of projective objects $P$ modulo the relations $[P \oplus Q]=[P]+[Q]$.  The multiplication on $K_0(\mathcal C)$ is defined as $[P] \cdot[Q]=[P \otimes Q]$ for any projective objects $P$ and $Q$ in $\mathcal C$. 

\end{definition}

In general, $K_0(\mathcal{C})$ can be a ring without a unit (rng).

\subsubsection{Tensor ideals}

For applications to quantum invariants of links and topological quantum field theories, we are interested in tensor ideals with modified traces in finite tensor categories.

\begin{definition}

Given an FTC $\mcC$, 
    \begin{enumerate}
\item A subcategory $\mcI \subseteq \mcC$  is called a weak tensor ideal if
for any $V \in \mcC$  and $X\in \mcI$, $V\otimes X\in \mcI$.
\item A weak tensor ideal  $\mcI \subseteq \mcC$ is called a tensor ideal if for any pair of objects $Y_1, Y_2$ in $\mcC$ such that $Y_1\oplus Y_2 \in \mcI$, then $Y_1,Y_2\in \mcI$.
    \end{enumerate}
\end{definition}

\subsubsection{Projective Auslander algebras}

Given an FTC $\mcC$ and a full list $\{P_i\}_{i\in I}$ of its projective indecomposable module representatives, we define the projective Auslander algebra $A_\mcC$ to be the algebra $A_\mcC=\Hom (P,P),$ where $P=\oplus_{i\in I}P_i$ with the composition of morphisms as the algebra multiplication.

For a fusion category, the projective Auslander algebra has no essentially new information than the vector space as it is a direct sum of $1$-dimensional algebras $\mbbC$.  But for general FTCs $\mcC$, their projective Auslander algebra can be non-commutative and very interesting as we see here for $\Rep(K_m)$.

For the Nicholas Hopf algebra $K_m$, their projective Auslander algebras are Hopf algebras.  It would be interesting to know when this is the case in general.

\subsection{Nichols Hopf algebras, their projective fusion rings and Auslander algebras}

\subsubsection{Nichols Hopf algebras $K_m$}

Let $m\in \mathbb{Z}_+$ be a positive integer. The Nichols Hopf algebra $K_m$ is the $2^{m+1}$-dimensional complex Hopf algebra with algebra generators $K, \xi_1, \dots, \xi_m$ and relations:

\begin{equation}\label{Kn-pres}
 K^2 = 1,\hspace{1em} \xi_i^2 = 0,\hspace{1em}K\xi_i =-\xi_i K,\hspace{1em}\xi_i\xi_j=-\xi_j\xi_i \hspace{1em}
\textrm{for}\; i, j=1,\dots,\hspace{1em} m\;\;\;
\textrm{and}\; i\neq j. 
\end{equation}

The comultiplication $\Delta:K_m \to K_m \otimes K_m$, counit $\epsilon:K_m \to\bC$ and antipode $S:K_m \to K_m$ are given as follows: 

\begin{align*}
\Delta(K) &= K\otimes K, & \epsilon(K) &= 1, & S(K) &= K,\\
\Delta(\xi_i) &= K\otimes \xi_i + \xi_i\otimes 1, & \epsilon(\xi_i) &= 0, & S(\xi_i) &= -K\xi_i.
\end{align*}

Note that $\mcK_1$ is the 4-dimensional Sweedler Hopf algebra, and the Nichols Hopf algebra $K_m$ is the crossed product of the exterior algebra of an $m$-dimensional complex vector space $E$ with the $\bZ/2\bZ$-group algebra $K_m \cong\Lambda^* E\rtimes \bC[\bZ/2\bZ]$.   The exterior algebra $\Lambda^* E$ itself is not a Hopf algebra.

\subsubsection{Finite tensor categories $\Rep(K_m)$ and their projective ideals} \label{sec: Projectives of Kn}

The representation categories $\Rep(K_m)$ of the Nicholas Hopf algebra $K_m$ are interesting non-semisimple FTCs.
There are two non-isomorphic irreducible $K_m$-modules $\{V(0), V(1)\}$ of dimension=1, and each has a projective cover $\{P(0), P(1)\}$ of dimension=$2^{m}$. Thus, the total rank of $\Rep(K_m)$ is 4 for all $m\geq 1$. 
Their full fusion rules are 
$$V(i)\otimes V(j)=V(i+j),\;\;\; V(i)\otimes P(j)=P(i+j) $$
$$P(i)\otimes P(j)=2^{m-1}P(0)\oplus 2^{m-1}P(1)$$ for all $i,j=0,1$ and the addition is mod 2.
 
The projective modules of $\Rep(K_m)$ form a full subcategory, which is a tensor ideal. 

\subsubsection{Projective Auslander algebras of $K_m$}

\begin{proposition}
The projective Auslander algebra of $\Rep(K_m)$ is isomorphic as an algebra to $K_m$ for each $m\geq 1$.
\end{proposition}
\begin{proof}
Let $P_0$ and $P_1$ denote $P(0)$ and $P(1)$ respectively for $K_m$.  Let $A$ denote the Auslander algebra: $A=\End_{K_m}(P_0 \oplus P_1)$. Let $e_0=\frac{1+K}{2}$ and $e_1=\frac{1-K}{2}$. Let $W$ denote the set of elements of $K_m$: $\xi_{i_1}\xi_{i_2} \cdots \xi_{i_k}$ with $1\leq i_1< i_2 \cdots <i_k$ for $0\leq k\leq m$. Then $We_i$ forms a basis for $P_i$. Let $p(n)$ denote the parity function $p(n)=n$ mod $2$. Then, note that for $r,s\in \{0,1\}$, as vector spaces, \begin{align*}
    \hom_{K_m}(P_r, P_s) & \simeq \hom_{K_m}(P_r\otimes P_s^*, 1)\\
    &\simeq\hom_{K_m} (2^{m-1} P_0 \oplus 2^{m-1} P_1,1)\\
    &\simeq\hom_{K_m} (2^{m-1} P_0, 1),
\end{align*} which is dimension $2^{m-1}$. Further, note that $\phi\in \hom_{K_m}(P_r, P_s)$ is entirely determined by the image of $e_r$ since $\phi(we_r)= w\cdot \phi(e_r)$. Further, note that for a map $\phi: P_r \to P_s$, \begin{align*}
    \phi(e_r)&= \sum\limits_{w\in W} \lambda_w w\cdot e_s
\end{align*}
but since $e_r \cdot e_r=e_r$, then 
\begin{align*}
    \sum\limits_{w\in W} \lambda_w w\cdot e_r &=\sum\limits_{w\in W} \lambda_w e_r \cdot (w\cdot e_s)\\
    &= w\cdot (e_{r+ p(|w|)}) e_s,
\end{align*} then for $r=s$, the odd length terms must be zero and for $r=s+1$, the even length terms are zero. 
Thus, there are $\sum\limits_{i=1}^{\lfloor \frac{m}{2} \rfloor} \begin{pmatrix}
    m\\
    2i
\end{pmatrix}= 2^{m-1}$ possible non-zero coefficients, so each must be possible by the dimension of the space of algebra homomorphisms. Further, for $w\in W$, let $\phi^w_r: P_r \to P_{r+p(|w|)}$ denote the map $e_r \mapsto we_{r+p(|w|)}$ and let $\phi_r^i: P(r) \to P(1+r)$ where $e_r\mapsto \xi_i e_{1+r}$. First, note that the set $\{\phi^w_r\}$ generate all possible maps from $P_r \to P_s$, for fixed $r$ and either choice of $s$, so the Auslander algebra is generated by the set $\{\phi_r^w\}$ with $r\in \mbbZ/2\mbbZ$ and $w\in W$. However, for non-empty word, note each $\phi_r^w$ is generated as a product $$\phi_r^w = \phi_{r + p(|w|)}^{w_1} \cdots \circ \phi_{r+1}^{w_{|w|-1}} \circ \phi_r^{w_{|w|}}$$ where $w_k$ is the $k$th term in $w$. Let $\iota_r$ denote $\phi_r^{\emptyset}$. Thus, the set $\{\phi_r^i, \iota_r \}$ with $r\in \{0,1\}, w\in W$ generate $A$. The product relations of $A$ are \begin{align*}
    \phi^i_r \phi_{1+r}^i=0 \tab \phi_r^i \phi_r^j=0 \tab \phi^i_r \phi^j_{1+r} = -\phi_{r}^j \phi^i_{1+r} \tab \phi_r^i \iota_r = \phi^i_r = \iota_{1+r} \phi^i_r\\
    \iota^2_r=\iota_r \tab \iota_r \iota_{1+r}=0 \tab \iota_r \phi^i_r= 0 = \phi^i_r \iota_{1+r},
\end{align*} which generates a $2^{m+1}$ dimensional algebra with complex basis given by $e_0, e_1, \phi_r^w$ where $w\in W$ and $r\in \{0,1\}$.\\
Define the map $F:A \to K_m$ by $F(\iota_r)= e_r$ and $F(\phi^i_r)=\xi_i e_r$. First, note that each of the above relations in $A$ are satisfied by the images in $K_m$ since $e_r \xi_i= \xi_i e_{1+r}$. Thus, $F$ induces an algebra homomorphism from $A\to K_m$. Also, for each $w\in W$, $F(\phi_{0}^w + \phi_1^w)=w$ and $F(\phi_{0}^w-\phi_{1}^w)=Kw$, so the map is surjective. Further, since both spaces are of the same dimension, it must be an isomorphism of algebras. 
\end{proof}

\section{Subcategories as complete invariants and reconstruction}\label{sec: sub}

There are in general continuous families of indecomposable objects in FTCs (see e.g. \cite{chen2014green} and Sec. \ref{sec: K2}), hence in general it is challenging to present an FTC with a finite set of matrices.  Our strategy is to look for subcategories of FTCs which are complete invariants and amenable to a presentation with finitely many matrices.  Then the full FTC can be reconstructed in a unique way up to equivalence from such a full subcategory. A natural choice seems to be the quasi-dominated sub-category \cite{turaev2010quantum, chang2024modular}, but unfortunately such categories can also be unwieldy with uncountably many indecomposable modules (see Sec. \ref{sec: K2}).  Instead we realize that the full projective subcategory is already a complete invariant, even though it makes the reconstruction of the FTC explicitly harder as the simple objects are only implicitly given. 

\subsection{Full projective subcategory as a complete invariant}  

Given an abelian category $\mcC$ with enough projectives, we will show that the full projective subcategory $\mcP_\mcC$ of $\mcC$ is a complete invariant of $\mcC$ in the sense that $\mcC$ is uniquely determined by $\mcP_\mcC$ up to equivalence.

\subsubsection{Full projective subcategory of an abelian category with enough projectives} 

For a finite tensor category $\mcC$, let $\mcP_\mcC$ denote its full projective subcategory. Since $\mcP_\mcC$ need not be unital nor abelian, we wish to consider a category with objects of this form. Thus, let $\mcB_R$ denote the category with objects additive tensor subcategories of finite tensor categories with maps right exact tensor functors up to natural equivalence. Let $\mcA$ denote the full subcategory of $\mcB_R$ consisting of finite tensor categories. \\
More generally, suppose $\mcC$ is an abelian category with enough projectives and possibly some additional structures, and let $\mcP_\mcC$ denote its full subcategory of projectives. Let $\mcB_R$ be the category with objects additive categories with the same structure and morphisms natural equivalence classes of right exact functors preserving this structure, and let $\mcA$ denote the full subcategory of $\mcB_R$ with abelian objects.   \\
\begin{lemma}  \cite[Proposition 1.2]{RepinAbelianCat}
Given a category $\mcC$ in $\mcA$ with projective subcategory $\mcP_\mcC$, let $\iota$ be the inclusion of $\mcP_\mcC$ into $\mcC$.
Then, $(\iota, \mcC)$ is the localization (or reflection) of $\mcP_\mcC$ in $\mcA$.
\end{lemma}
\begin{proof}
Given abelian $\mcD$ and structure-preserving right-exact map $F:\mcP\to \mcD$, we must define a unique right exact structure-preserving extension $F: \mcC\to \mcD$.\\
We extend this map $F$ to a map $F'$ as if constructing the zeroth derived functor. First, for each object $A\in \mcC$, there exists an exact projective resolution $P_A^\bullet\to A \to 0$, so we define $F'(A)=L_0 F(P_A^\bullet)$ and for a map $f: A\to B$, if $P_A^\bullet \to A \to 0$ is a projective resolution of $A$ and $P_B^\bullet \to B \to 0$ is a projective resolution of $B$, then there exists a chain map $\overline{f}: P_A^\bullet \to P_B^\bullet$ that is unique up to chain homotopy inducing $f$ as a map on $H_0(\overline{f})$. Then, define $F'(A)=H_0 (F(P_A^\bullet))$ then $F'(f)=L_0 F(\overline{f})$. For each diagram,
    \[ \begin{tikzcd} P_A^\bullet \arrow{r}{\phi} \arrow{d}{\overline{f}} & A \arrow{d}{f} \arrow{r} & 0 \\ P_B^\bullet \arrow{r}{\phi'}  & B \arrow{r} & 0,
\end{tikzcd}
\] we obtain an exact diagram with exact rows     \[ \begin{tikzcd} F(P_A^\bullet) \arrow{d}{F(\overline{f})}  \\F(P_B^\bullet) 
\end{tikzcd}
\]
First, note that since for any projective resolution, the induced functors are naturally isomorphic, restricted to the projective subcategory, $F'$ is naturally isomorphic to $F$. Suppose $G$ was another extension of $F$. Then, since $G$ is right exact, $G$ must be be naturally isomorphic to $L_0 G(P_A^\bullet)$ which is naturally equivalent to $G$, but since $F$ and $G$ are naturally equivalent functors on $\mcP$ and both $L_0 G$ and $L_0 F$ are entirely determined by action on the projective subcategory, then $L_0 G$ is naturally equivalent to $L_0 F$, which is naturally equivalent to $F$. Therefore, the choice of extension was unique up to natural equivalence.
\end{proof}

Now we prove the main result of this section.

\begin{theorem}
Let $\mcC$ be an abelian category with any additional structures and let $\mcP$ be its full projective subcategory. Then $\mcC$ is the unique abelian completion of $\mcP$ up to the corresponding natural equivalence. Namely, the full projective subcategory is a complete invariant of a finite tensor category.
\end{theorem}

\begin{proof}
Let $\mcB_R$ denote the category of additive objects of the structure of $\mcC$ and morphisms equivalence classes right exact structure preserving maps and let $\mcA$ denote the full subcategory of $\mcB_R$ of abelian objects. Suppose $\mcC$ and $\mcC'$ are both abelian categories with full projective subcategory $\mcP$. Then, both $\mcC$ and $\mcC'$ are reflections of $\mcP$ in $\mcA$. Therefore, let $\iota,\iota'$ denote the inclusion functors from $\mcP$ into $\mcC$ and $\mcC'$ respectively. Then, by the definition of reflection, there exist two right exact structure-preserving functors $F:\mcA\to\mcA'$ and $G:\mcA'\to \mcA$ such that $F\circ \iota\simeq \iota'$ and $G\circ \iota'\simeq \iota$. Then, by post-composing $F$ with natural transformation $\eta$ and $G$ with $\eta'$, we may assume that $F\circ \iota = \iota'$ and $G\circ \iota'=\iota$. Therefore, for objects in $P\in \mcP$, $G\circ F(\iota(P))=P$ and $F\circ G(\iota'(P))=P$. Thus, on the projective subcategories, both $G\circ F$ and $F\circ G$ are the identity functor, so since $F\circ G\simeq L_0 (F\circ G)$ which is the identity transformation, so $F\circ G$ is naturally isomorphic to the identity, and by the same argument, as is $G\circ F$. Thus, $\mcA$ and $\mcA'$ are isomorphic via structure preserving maps.
\end{proof}

\subsection{Quasi-dominated subcategory as a complete invariant}

By definition, an FTC $\mcC$ is an abelian category, therefore all kernels and cokernels are in the category. This would lead to the Green ring, which could be very complicated in the non-semisimple case.  Instead we may consider only a sub-category $\mcC_{QD}$ of $\mcC$, not necessarily abelian or monoidal, that consists of only objects that are quasi-dominated by the simple objects. 

An object $X$ of $\mcC$ is quasi-dominated by some simple objects $\{X_i\}$ of $\irRC$ if there exist a family of morphisms $\alpha_{i,m_i}: X_i\rightarrow X, \beta_{i,m_i}: X\rightarrow X_i$ such that
$$ \textrm{Id}_X-\sum_{i, m_i}\alpha_{i,m_i}\beta_{i,m_i}$$ is a negligible morphism \cite{turaev2010quantum}. It is clear that every negligible object is in $\mcC_{QD}$, every simple object is in $\mcC_{QD}$, and finite direct sums of neglible and simple objects are in $\mcC_{QD}$.

\subsubsection{Characterization of quasi-dominated subcategory}

Let $\mcN_\mcC$ be the set of negligible objects, and $\mcS_\mcC$ be the set of simple objects of a finite tensor category $\mcC$, respectively.  Then we will show in this subsection:

$$\overline{\mcN_\mcC\cup \mcS_\mcC}^{\textrm{AC}}=\mcC_{QD}. $$

\begin{theorem} \label{theorem: CQD classification}
Let $\mcA$ be a $\mbbC$-linear abelian category with a function $tr_X: \hom(X,X)\to \mbbC$ satisfying cyclicity. Then, the quasi-dominated subcategory is the additive completion of the set of negligible and simple objects.\end{theorem}
\begin{proof}
First, suppose $X= \bigoplus\limits_k M_k \in \mcA_{QD}$. Let $\{X_i\}$ denote the set of simple objects. Then, there exist maps $\alpha_{i,m_i}: X_i\rightarrow X, \beta_{i,m_i}: X\rightarrow X_i$ such that $ \textrm{Id}_X-\sum_{i, m_i}\alpha_{i,m_i}\beta_{i,m_i}$ is a negligible morphism. We claim that the maps $\tilde{\alpha}_{i,m_i} =\alpha_{i,m_i} \circ \iota$ and $\tilde{\beta}_{i,m_i}= \pi \circ \beta_{i,m_i}$, where $\iota$ is the inclusion of $M_1$ into $X$ and $\pi$ is the projection from $X$ to $M_1$, satisfy the desired properties such that $M_1\in \mcA_{QD}$. Let $h: M_1 \to M_1$ be a morphism. Then, first, since $\iota \circ h \circ \pi: M\to M$, then $$tr\big((\textrm{Id}_X-\sum_{i, m_i}\alpha_{i,m_i}\beta_{i,m_i}) \circ (\iota \circ h \circ \pi)\big)=0.$$  Therefore,
    \begin{align*}
tr\big(\pi\circ (\textrm{Id}_X-\sum_{i, m_i}\alpha_{i,m_i}\beta_{i,m_i} ) \circ (\iota \circ h)\big)&=0\\
tr\Big( \big(\pi \circ \textrm{Id}_X \circ \iota - \pi \circ (\sum_{i, m_i}\alpha_{i,m_i}\beta_{i,m_i}) \circ \iota\big) \circ h\Big)&=0\\
tr\Big(\big(\textrm{Id}_{M_1}- \sum_{i, m_i}\tilde{\alpha}_{i,m_i}\tilde{\beta}_{i,m_i}\big) \circ h\Big) 
&=0.
\end{align*} Thus, by definition $\big(\textrm{Id}_{M_1}-\sum_{i, m_i}\tilde{\alpha}_{i,m_i}\tilde{\beta}_{i,m_i}\big)$ is negligible. Thus, since $\mcA_{QD}$ is closed under retracts, then it is idempotent complete. \\
Now suppose $X$ is an indecomposable object in $\mcA_{QD}$. Then, there exist maps $\alpha_{i,m_i}: X_i\rightarrow X, \beta_{i,m_i}: X\rightarrow X_i$ such that $ \textrm{Id}_X-\sum_{i, m_i}\alpha_{i,m_i}\beta_{i,m_i}$ is a negligible morphism. First, suppose that $\beta_{j,m_j}\alpha_{i,m_i}=0$ for all $(i,m_i), (j,m_j)$. Let $h:M\to M$. Then $$(\textrm{Id}_X-\sum_{i, m_i}\alpha_{i,m_i}\beta_{i,m_i} )(\textrm{Id}_X+\sum_{i, m_i}\alpha_{i,m_i}\beta_{i,m_i} ) h: M\to M,$$ so it has has no trace. But, $(\textrm{Id}_X-\sum_{i, m_i}\alpha_{i,m_i}\beta_{i,m_i} )(\textrm{Id}_X+\sum_{i, m_i}\alpha_{i,m_i}\beta_{i,m_i} )$ is just the identity on $M$, so this implies that $h$ has no trace, so by definition, $M$ is negligible.
Now, suppose $\beta_{j,m_j}\alpha_{i,m_i}\ne 0$ for some $(i,m_i),(j,m_j)$. Then, since the map is between two simple objects, it must be an isomorphism. Therefore, there exists an inverse $\sigma: X_j \to X_i$. Therefore, the short exact sequence $$0\to X_i \to X \to \coker \alpha_{i,m_i} \to 0$$ has a retract, namely $\sigma \circ \beta_{j,m_j}$, so the sequence splits. Thus, $X\simeq X_j$ since it was assumed to be indecomposable. 
\end{proof}

\subsection{Categorical reconstruction from projectives}

First, recall that for any FTC $\mathcal{C}$, since it is a finite abelian category, every object has a projective cover and an injective hull \cite[Remark 1.8.7]{etingof2004finite}.  Also, the injective objects and projective objects align. 

For any FTC $\mcC$, let $\mcP_\mcC$ denote its full subcategory of projective objects.  For any object in $\mcC$, it has an projective cover $P_1$ and injective hull $P_2$ in $\mcP_\mcC$, so it is the kernel of the cokernel of a map $P_1\to P_2$. Further, any kernel of a cokernel of a map $P_1\to P_2$ is an element of $\mcC$. Therefore, the objects of $\mcC$ are fully determined by the maps between projectives.  \\

We record the following characterization of simple objects.

\begin{lemma} If $\mcP$ is a full projective subcategory of a finite tensor category, then for a nonzero map $\phi: P_i \to P_k$ between  projective indecomposables, $\im \phi$ is a simple object if and only if for every map nonzero $\psi:P_l \to P_k$ and $j:P_k\to P$ such that $j\circ \psi=0$, then $j\circ \phi=0$ as well where $P_i,P_k,P_l$ are projective indecomposables and $P$ is projective. \end{lemma}

\begin{proof}
First, note that every simple object $V_i$ has indecomposable projective cover $P_i$ and injective hull $P_k$. Then, if the image of $\phi$ is a simple subobject of $P_k$ which is indecomposable, then it must be its socle, then it is contained in every subobject. Thus, for any $\psi$, $\im \phi \subset \im \psi$, so if $j\circ \psi=0$, then $j\circ \phi=0$. \\

Further, suppose $\phi$ satisfies this property for every map $\psi, j$. Then, since $P_k$ is projective, it contains a simple subobject $V_k$ which has projective cover which we call $P_l$. Thus, by considering $\psi$ to be the projective covering map composed with the injective map $V_k\to P_k$, and $j$ to be the composition of the cokernel of $\psi$ with the inclusion into its injective hull, then $\psi \circ j=0$. However, this implies that $\phi \circ j=0$, but since the inclusion into its injective cover is injective, then $\im \phi \subset \im \psi$, but since $\im \psi$ is simple, then so is $\im \phi$ and they are isomorphic. 
\end{proof}

Categorically, we can reconstruct an FTC $\mcC$ from its projective ideal $\mcP_\mcC$ as follows.

We recover the objects of $\mcC$ by considering the image of maps in $\mcP_\mcC$. Since each object has a projective cover  and an injective hull, then each object $A$ is isomorphic to the image of a map denoted $\phi_A$ in $\mcP_\mcC$. Further, since $\mcC$ is abelian, then the image of every map in $\mcP_\mcC$ must be in $\mcC$. Additionally, any map between two objects $A,B\in \mcC$ can be recovered as a map from the projective cover of $A$ to the injective hull of $B$ that descends to a map between $\im \phi_A \to \im \phi_B$. Further, any map $f$ from the projective cover of $A$ to the injective hull of $B$ that descends to a map between $\im \phi_A \to \im \phi_B$ defines a map between $A$ and $B$ via composition with the fixed isomorphisms $A\to \im \phi_A$ and $\im \phi_B\to B$. Furthermore, we recover tensor products and direct sums of $A$ and $B$ by the images of the tensor products and direct sums of $\phi_A$ and $\phi_B$. Finally, for three objects $A_1, A_2, A_3$ with projective covers $P_i$ and covering maps $\psi_i$, we recover the associativity constraint by 
\[ \begin{tikzcd}
(P_1 \otimes P_2) \otimes P_3 \arrow{r}{a_{P_1,P_2,P_3}} \arrow{d}{(\psi_1 \otimes \psi_2) \otimes \psi_3} & P_1\otimes(P_2\otimes P_3)  \arrow{d}{\psi_1 \otimes (\psi_2 \otimes \psi_3)} \\
(A_1 \otimes A_2) \otimes A_3 \arrow{r}{a_{A_1, A_2, A_3}}& A_1 \otimes (A_2 \otimes A_3)
\end{tikzcd}
\]
which serves as a definition since the left and right side of the diagrams are surjective, and the top arrow is an isomorphism, so there is only one possible map $a_{A_1, A_2, A_3}$.

\section{Indecomposable modules of Nicholas Hopf algebras and their doubles}\label{sec: nicholas}

In this section, we show how to determine the Green ring $r(K_m)$ of $K_m$ for even $m=2n$ from the Green ring $r(DK_n)$ of the Drinfeld double $DK_n$.

Let $K_{2n}$ be generated by $[K], [\xi_1],...,[\xi_n], [\overline{\xi}_1]=\xi_{n+1}, ..., [\overline{\xi}_n]=\xi_{n+n}$ and let $DK_n$ be generated by $\overline{K}, K, \xi_1,...,\xi_n, \overline{\xi}_1, ..., \overline{\xi}_n$.  Note that $K_{2n}=DK_n/<K-\bar{K}>$, and $\pi$ denotes the projection map $\pi: DK_n \to K_{2n}$ that identifies $\overline{K}$ and $K$.

The projection $\pi$ induces maps between the Green rings $r(K_{2n})$ and $r(DK_n)$.  Using these two maps, we identify the Green ring $r(K_{2n})$ as a sub-ring $r(DK_n)_0$ of $r(DK_n)$, where the two generators $\overline{K}$ and $K$ act the same way.

\subsection{From the Green ring $r(K_{2n})$ to the Green ring $r(DK_n)$}
\subsubsection{The map $\pi_*:  r(K_{2n})\to r(DK_n)_0$}
\begin{lemma} \label{lem: WD}
If $M\cong M'$ as $DK_n$ modules and $K.m=\overline{K}.m$ for all $m\in M$, then $K.m'=\overline{K}.m'$ for all $m'\in M'$.
\end{lemma}

\begin{proof}
Since $K-\overline{K}$ must annihilate $M$, it must annihilate any $DK_n$ module isomorphic to $M$.
\end{proof}
Therefore, the set of isomorphism classes of modules with this property defines a subset of $r(DK_n)$. Define $r(DK_n)_0$ to be $$r(DK_n)_0 = \{[M]\in r(DK_n) : K.m = \overline{K}.m \,\forall m\in M\}.$$ 

\begin{lemma} \label{lem: subring}
If $M$ and $M'$ are $DK_n$ modules, \begin{enumerate}[i.] 
    \item if $[M_1],[M_2]\in r(DK_n)_0$, then $[M_1\oplus M_2]\in r(DK_n)_0$
    \item if $[M_1],[M_2]\in r(DK_n)_0$, then $[M_1\otimes M_2]\in r(DK_n)_0$\end{enumerate}
\end{lemma}

\begin{proof}
i. Let $(m,m')\in M\oplus M'$. Then, \begin{align*}
    K.(m,m')&= (K.m, K.m')\\
    &= (\overline{K}.m, \overline{K}.m')\\
    &= \overline{K}.(m,m').
\end{align*} Thus, if $[M_1], [M_2]\in r(DK_n)_0$, then $[M_1\oplus M_2]\in r(DK_n)_0$. \\
ii. Let $m\otimes m'$ be a simple tensor in $M\otimes M'$. Then, \begin{align*}
    K.(m\otimes m')&= (K.m\otimes K. m')\\
    &= \overline{K}.m\otimes \overline{K}.m'\\
    &= \overline{K}.(m\otimes m').
\end{align*} Thus, since $K-\overline{K}$ is in the annihilator of all simple tensors, it is in the annihilator of the entire tensor product. Thus, if $[M_1],[M_2]\in r(DK_n)_0$, then $[M_1\otimes M_2]\in r(DK_n)_0$.
\end{proof}

Thus, since the identity module in $r(DK_n)$ is the trivial representation is in $r(DK_n)_0$, then Lemma \ref{lem: subring} implies that this is a subring of the Green ring of $DK_n$.
\begin{lemma} \label{lem: addmult}
Any $K_{2n}$ module homomorphism can be viewed as a $DK_n$ module homomorphism on the induced $DK_n$ module structure via the composition with the projection map $\pi: DK_n \to K_{2n}$ that identifies $\overline{K}$ and $K$. Therefore, any two modules that are isomorphic as $K_{2n}$ modules are isomorphic as $DK_n$ modules.
\end{lemma}

\begin{proof}
Let $\phi:M\to N$ be a $K_{2n}$ module homomorphism. Let $m\in M$ and $g$ is a generator in $\{\xi_1,...,\xi_n, \overline{\xi}_1,...,\overline{\xi}_n, K,\overline{K}\}$. 
Then, \begin{align*} g.\phi(m) &= [g].\phi(m)\\
 &= \phi([g].m) \\
 &= \phi(g.m).
 \end{align*}
Thus, $\phi$ is a $DK_{n}$ module homomorphism as well. Thus, $M\cong V_1 \oplus V_2$ as $DK_n$ modules as well. 
\end{proof}

Note that for all modules of this form, $K.m=\overline{K}.m$ for all $m\in M$ since $[K]=[\overline{K}]$. 

Define $\pi_*: r(K_{2n}) \to r(DK_n)_0$ to be the isomorphism class of viewing $K_{2n}$ modules as $DK_n$ via composition with the projection $\pi$. By Lemma \ref{lem: addmult}, $\pi_*: r(K_{2n}) \to r(DK_n)_0$ is a ring homomorphism. \\

\subsection{From the Green ring $r(DK_n)$ to the Green ring $r(K_{2n})$}

\subsubsection{The map $\pi^*: r(DK_n)_0\to r(K_{2n})$}
We now define $\pi^*: r(DK_n)_0 \to r(K_{2n})$. Let $\phi: DK_n \to \text{End}(M)$ be a finitely-generated  $DK_n$-module such that $\phi(K) = \phi(\overline{K})$. Define $\pi^*(\phi)([\xi_i]) = \phi(\xi_i)$, $\pi^*(\phi)([\overline{\xi}_i]) = \phi(\overline{\xi}_i)$ and $\pi^*(\phi)([K])=\phi(K)$. To prove that this defines a $K_{2n}$ module structure on $M$, it suffices to show that the maps $\pi^*(\phi)([\xi_i])$, $\pi^*(\phi)([\overline{\xi}_i])$, and $\pi^*(\phi)(K)$ satisfy the properties of \cite[Equation 15]{chang2024modular}. First, note that \begin{align*} (\pi^*(\phi)(K))^2&= (\phi(K))^2 \\&= \phi(K^2)\\ &=\phi(1).\end{align*} Similarly, for both $\xi_i$ and $\overline{\xi_i}$,
\begin{align*} (\pi^*(\phi)(\xi_i))^2&= (\phi(\xi_i))^2 \\&= \phi(\xi_i^2)\\&= 0.\end{align*} Next, \begin{align*} \pi^*(\phi)(\xi_i K + K\xi_i) &= \pi^*(\phi)(\xi_i) \phi(K) + \phi(K) \phi(\xi_i)\\
&= \phi(\xi_i K +K \xi_i) = \phi(0). \end{align*} Further, \begin{align*} \pi^*(\phi)(\overline{\xi}_i K + K\overline{\xi}_i) &= \pi^*(\phi)(\overline{\xi}_i) \phi(K) + \phi(K) \phi(\overline{\xi}_i)\\
&=\pi^*(\phi)(\overline{\xi}_i) \phi(\overline{K}) + \phi(\overline{K}) \phi(\overline{\xi}_i)\\
&= \phi(\overline{\xi}_i \overline{K} + \overline{K} \overline{\xi}_i)= \phi(0). \end{align*} Finally, by the same argument, $\pi^*(\phi)(\xi_i) \pi^*(\phi)(\xi_j)=-\pi^*(\phi)(\xi_j)\pi^*(\phi)(\xi_i)$ and $$\pi^*(\phi)(\overline{\xi}_i)\pi^*(\phi)(\overline{\xi}_j)=-\pi^*(\phi)(\overline{\xi}_j)\pi^*(\phi)(\overline{\xi}_i).$$ The last relation to check is \begin{align*} 
\pi^*(\phi)(\overline{\xi}_i) \pi^*(\phi)(\xi_j)+ \pi^*(\phi)(\xi_j)\pi^*(\phi)(\overline{\xi}_i) &= \phi(\overline{\xi}_i \xi_j + \xi_j \overline{\xi}_i) \\
&= \phi(\delta_{i,j} (1-K\overline{K})) \\
&=\delta_{i,j} (\phi(1) - \phi(K)^2) \\
&= \delta_{i,j} (\phi(1) - \phi(K^2))\\
&=\phi(0). \end{align*}
Thus, $\pi^*(\phi)$ defines a well-defined $K_{2n}$ module structure on a module in $r(DK_n)_0$.

\begin{lemma} \label{lem: Dmultadd}
If $M\simeq V_1 \oplus V_2$ as $DK_n$ modules with $[M]\in r(DK_n)_0$, then $[V_1],[V_2]\in r(DK_n)_0$ and $M\simeq V_1\oplus V_2$ as a $K_{2n}$ module where $[K]$ acts as $K$. Similarly, if $M\simeq V_1 \otimes V_2$ as $DK_n$ modules and $[M], [V_1], [V_2]\in r(DK_n)_0$, then $M\simeq V_1\otimes V_2$ as $K_{2n}$ modules as well.
\end{lemma}

\begin{proof}
First, note that if $M\simeq V_1 \oplus V_2$, then $K-\overline{K}$ is in the annihilator of every element in $M$, so it is also in the annihilator of every element in $V_1\times \{0\}$ and $\{0\} \times V_2$ which imply that $K-\overline{K}$ annihilates both $V_1$ and $V_2$ by the definition of a direct sum module. Therefore, both $[V_1],[V_2]\in r(DK_n)_0$.
 
Suppose $\phi$ is a $DK_n$ module isomorphism from $M$ to $V_1\oplus V_2$ or $V_1 \otimes V_2$ with $[M], [V_1], [V_2]\in r(DK_n)_0$. Let $m\in M$ and $g$ is a generator in $$\{[\xi_1],...,[\xi_n], [\overline{\xi}_1],...,[\overline{\xi}_n], [K]\}.$$ 
Then, \begin{align*} [g].\phi(m) &= g.\phi(m)\\
 &= \phi(g.m) \\
 &= \phi([g].m).
 \end{align*}
Thus, $\phi$ is a $K_{2n}$ module isomorphism as well.  
\end{proof}

\begin{rmk} 
Lemma 5 implies that if $M$ is a $DK_n$ module such that $[M]\in r(DK_n)_0$, then it is a decomposable module if and only it splits as a direct sum of two $DK_n$ modules $M_1\oplus M_2$ with both $[M_1],[M_2]\in r(DK_n)_0$. 
\end{rmk} \label{rem: indecomp.}

\begin{theorem}
Thus, the ring $R(K_{2n})$ is isomorphic to the subring $r(DK_n)_0\subset r(DK_n)$. Further, the indecomposable $K_{2n}$ modules are precisely the indecomposable $DK_n$ modules where $K$ and $\overline{K}$ define the same endomorphism. 
\end{theorem}

\begin{proof}
    
By Lemma \ref{lem: Dmultadd} $\pi^*: r(DK_n)_0 \to r(K_{2n})$ is a ring homomorphism. Furthermore, $\pi^*\circ \pi_*([M]) = [M]$ for any $[M]\in r(K_{2n})$ since there are representatives of the two modules that have the same underlying abelian group and agree on the generators. Furthermore, $\pi_*\circ \pi^* ([M])=[M]$ for any $[M]\in r(DK_n)_0$ for the same reason. Thus, the ring $R(K_{2n})$ is isomorphic to the subring $r(DK_n)_0\subset r(DK_n)$. Note that Remark \ref{rem: indecomp.} implies that the indecomposable $K_{2n}$ modules are precisely the indecomposable $DK_n$ modules where $K$ and $\overline{K}$ define the same endomorphism. 
\end{proof}

\section{Green ring of $K_2$ and its tensor ideals}\label{sec: K2}

In this section, we determine the Green ring $r(K_2)$ of $K_2$ and also find their tensor ideals.  The deduction of the Green ring $r(K_2)$ from the Green ring $r(DK_1)$ of $DK_1$ in \cite{chen2014green} is based on our results in last section.  We will follow the notation in \cite{chen2014green}. 

In this section, $r,r'\in \mathbb{Z}_2=\{0,1\}$ with mod 2 addition, $s,n\in \mathbb{Z}_+$ are positive integers, and $\eta \in \mathbb{C}P^1=\mathbb{C}\cup \{\infty\}$ if not specified explicitly.  The parity function $p(n)=n$ mod $2$.

\subsection{Green ring $r(K_2)$ of $K_2$} 

\subsubsection{Hopf algebra $DK_1$}

The Hopf algebra $DK_1$ is a $16$-dimensional algebra generated as an algebra by $a,b,c,d$ such that 
$$a^2=0, d^2=0, b^2=1, c^2=1, ad+da=1-bc,$$ and 
$$ab=-bc,ac=-ca, bd=-db, cd=-dc,bc=cb.$$

The generators $b, c$ are group-like elements, and the coalgebra and antipode are further given as  
$$\Delta (a)=a\otimes b+1\otimes a, \Delta (d)=d\otimes c+1\otimes d, S(a)=-ab, S(d)=-dc.$$

\subsubsection{Indecomposable modules of $DK_1$}

The indecomposable modules of $DK_1$ consist of 4 types: irreducible representations $V(r)$ and the projective covers $P(r)$ of $V(r)$, and two infinite families of indecomposable modules of Syzygy type  $\Omega^sV(r), \Omega^{-s}V(r), s\in \mathbb{Z}_+$ and M-type  $M_n(r,\eta), n\in \mathbb{Z}_+, \eta \in \mathbb{C}P^1.$  There are also two Steinberg modules $V(2,r)$, which are both irreducible and projective.

\subsubsection{Indecomposable modules of $K_2$}

The indecomposable modules of $K_2$ are those of $DK_1$ such that $K, \bar{K}$ act the same way. The excluded indecomposable modules of $DK_1$ are the Steinberg modules $V(2,r)$ as we show below.

Consider the indecomposable modules as presented in \cite{chen2014green}. First, note that by the explicit formula for $V(1,r)$, each $V(r)=V(1,r)\in r(DK_1)_0$. Also, each $M_1(r,\infty)$ and $M_1(r,\eta)$ are in $r(DK_1)_0$. Furthermore, $P(r)\in r(DK_1)_0$ for $r\in \mbbZ/2\mbbZ$. However, $V(2,r)\notin r(DK_1)_0$. Finally, note that the minimal projective resolution of $V(r)$ is given by $$\cdots \rightarrow 4P(r+1) \rightarrow 3P(r) \rightarrow 2P(r+1) \rightarrow P(r) \rightarrow V(r) \rightarrow 0$$ and the minimal injective resolution of $V(r)$ is given by $$0\rightarrow V(r) \rightarrow P(r) \rightarrow 2P(r+1) \rightarrow 3P(r) \cdots $$
For $k\geq 1$, Since $\Omega^k V(r)$ is the kernel of a $DK_1$ map from $kP(r+p(k)-1)$  and $K-\overline{K}$ is in the annihilator of every element, then $K-\overline{K}$ is in the annihilator of every element in the kernel as well, so $\Omega^k V(r) \in r(DK_1)_0$. Further, since $\Omega^{-k} V(r)$ is the cokernel of a map from into $kP(r+p(k)-1)$, then since $K-\overline{K}$ annihilates all elements of $P(r)$, then it will annihilate all elements of the cokernel. Thus, $\Omega^{-k} \in r(DK_1)_0$ as well. Also, note that while $V(2,0)$ is not in $r(DK_1)_0$, $V(2,0)\otimes V(2,0)=P(1)$ is in $r(DK_1)_0$. Finally, note that $M_n(r,\eta)$ is a submodule of $nP(r)$, so it is also in $r(DK_n)_0$.

We separate the indecomposable modules of $K_2$ into three types: 
the simple-projective $\{V(r), P(r)\}$, the Syzygy type $\{\Omega^{s}V(r),\Omega^{-s}V(r)\}$, and the continuous M-type $\{M_n(r,\eta)\}$.  The dimensions of these modules are:

$$\dim V(r)=1, \dim P(r)=4, \dim \Omega^{\pm s}V(r)=2s+1, \dim  M_n(r,\eta)=2n. $$

We deduce the following from \cite{chen2014green}.

\begin{theorem}

The full set of indecomposable modules of $K_2$ is
$$\{V(r), P(r), \Omega^{s}V(r),\Omega^{-s}V(r), M_n(r,\eta)\},$$ where $r\in \mbbZ_2, s, n\in \mathbb{Z}_+, \eta \in \mbbC P^1$.

The fusion rules of all the indecomposable modules of $K_2$ are as follows. The first set is the full fusion ring of simples $V(r)$ and projectives $P(r)$.

\begin{enumerate}

\item $V(r)^*=V(r), P(r)^*=P(r),$
\item $V(r)\otimes V(r')=V(r+r'),\;\;\; V(r)\otimes P(r')=P(r+r'),\;\;\; P(r)\otimes P(r')=2P(0)\oplus 2P(1).$
\end{enumerate}

The rest of the fusion rules are:
\begin{enumerate}
\item $(\Omega^{s}V(r))^*=\Omega^{-s}V(r), (M_n(r,\eta))^*=M_n(r+1,\eta),$
\item $V(r)\otimes \Omega^{s}V(r')=\Omega^{s}V(r+r'), V(r)\otimes \Omega^{-s}V(r')=\Omega^{-s}V(r+r'), V(r)\otimes M_n(r',\eta)=M_n(r+r',\eta),$
\item $P(r)\otimes \Omega^{s}V(r')=P(r)\otimes \Omega^{-s}V(r')=\oplus sP(r+r'+p(s+1))\oplus (s+1)P(r+r'+p(s)), P(r)\otimes M_n(r',\eta)=\oplus nP(0)\oplus nP(1)$.
\item $\Omega^{s}V(r)\otimes M_n(r',\eta)=M_n(r+r'+p(s),\eta)\oplus snP(r+r'+p(s+1))$, $\Omega^{-s}V(r)\otimes M_n(r',\eta)=M_n(r+r'+p(s+1),\eta)\oplus snP(r+r'+p(s))$
\item $\Omega^{s}V(r)\otimes \Omega^{n}V(r')=\Omega^{s+n}V(r+r')\oplus snP(r+r'+p(s+n))$, $\Omega^{s}V(r)\otimes \Omega^{-n}V(r')=\Omega^{s-n}V(r+r')\oplus (s+1)nP(r+r'+p(s+n+1))$ if $s\geq n$, and $\Omega^{s}V(r)\otimes \Omega^{-n}V(r')=\Omega^{s-n}V(r+r')\oplus (n+1)sP(r+r'++p(s+n+1))$ otherwise.
\item $M_s(r,\eta)\otimes M_n(r',\eta')=\oplus sn P(r+r')$ if $\eta\neq \eta'$, $M_s(r,\eta)\otimes M_n(r',\eta)=\oplus s(n-1) P(r+r')\oplus M_s(0,\eta)\oplus M_s(1,\eta) $ if $s\leq n$.

\end{enumerate}

\end{theorem}

\subsubsection{Green ring of $DK_1$}

Let $\mbbZ[X]$ be the polynomial algebra generated over $\mbbZ$ generated by $$X=\{g_1,x_1,y_1,z_1, X'_{n,\eta}| n\geq 1, \eta\in \mbbC P^1\}$$ and let $J$ be the ideal generated by the following polynomials as given in \cite{chen2014green}:

$g_1^2-1$, $x_1^3-2x_1(1+g_1)$, $x_1(y_1-1-2g_1)$, $x_1(y_1-z_1)$, $y_1z_1-1-2x_1^2$, $x_1X'_{n,\eta}-n(1+g_1)x_1$, $y_1X'_{n,\eta}-ng_1x_1^2-g_1X'_{n,\eta}$, $z_1X'_{n,\eta}-nx_1^2-g_1X'_{n,\eta}$, $X'_{n,\eta}X'_{s,\alpha}-nsg_1x_1^2$, $X'_{n,\eta}X'_{t,\eta}-n(t-1)g_1x_1^2-X'_{n,\eta}-g_1X'_{n,\eta}$,where $n,s,t\geq 1$ with $t\geq n$, $\eta, \alpha \in \mbbC P^1$ with $\eta\neq \alpha$.

Then the Green ring of $DK_1$ is  $r(DK_1)\cong\mbbZ[{X}]/J$.

\subsubsection{Green ring of $K_2$}

Let $\mbbZ[\tilde{X}]$ be the subring of $\mbbZ[X]$ generated by $$\tilde{X}=\{g_1,x_1^2,y_1,z_1, X'_{n,\eta}| n\geq 1, \eta\in \mbbC P^1\}.$$

It follows that the Green ring of $K_2$ is:

\begin{theorem}
As a ring, $r(K_2)\cong\mbbZ[\tilde{X}]/J$.
\end{theorem}

\begin{proof}
Let $g=[V(1)]$, $x=[V(2,0)]$, $y=[\Omega V(0)]$, $z=[\Omega^{-1} V(0)]$, and $X_{n,\eta}= M_n(0,\eta)$. Let $f: \mbbZ[X] \to r(DK_1)$ given by $$f(g_1)=g \tab f(x_1)=x \tab f(y_1)=y \tab f(z_1) = z \tab f( X'_{n,\eta}) = X_{n,\eta}$$ as in \cite{chen2014green}. Also, since $f(g)=0$ for all $g\in J$, $f$ induces a ring epimorphism $\overline{f}:\mbbZ[X]/J\to r(DK_1)$ such that $\overline{f}(\overline{u})=f(u)$ for all $u\in \mbbZ[X]$. By \cite[Theorem 3.9]{chen2014green}, this map is an isomorphism of rings. First, we show $r(DK_1)_0\subseteq \overline{f}(\mbbZ[\tilde{X}]/J)$. Let $$f_n = a_n(1+g)-\frac{n(n-1)}{2} g^n$$ where $a_n\in \mbbZ$ is defined by \cite[Lemma 3.2]{chen2014green}. First, note that by \cite[Lemma 3.3 and Lemma 3.4]{chen2014green} $$y^n-f_n x^2=[\Omega^nV(0)] \in f(\mbbZ[\tilde{X}])$$ and $$z^n-f_n x^2=[\Omega^{-n}V(0)] \in f(\mbbZ[\tilde{X}]).$$ Also, $$f(x'^2)=[P(r)] \tab f(g)=[V(1)] \tab f(g' X'_{n,\eta})=[M_n(1,\eta)]$$ are all elements of $\overline{f}(\mbbZ[\tilde{X}]/J)$. Since $V(1) \otimes \Omega^k V(0) \simeq \Omega^k V(1)$ for any $k\in \mbbZ$ and $V(1) \otimes M_n(1,\eta)\simeq M_n(0,\eta)$, then the isomorphism classes of all indecomposable modules where $K$ and $\overline{K}$ act in the same way, which generate $r(DK_1)_0$ as a ring, are in $\overline{f}(\mbbZ[\tilde{X}]/J)$. Thus, $r(DK_1)_0\subseteq \overline{f}(\mbbZ[\tilde{X}]/J)$. 

Further, since for each of these generators, $f(g'), f(x'^2), f(y'), f(z')\in r(DK_1)_0$ and for each $X'_{n,\eta}$, $f(X'_{n,\eta})=X_{n,\eta}=[M_n(0,\eta)]\in r(DK_1)_0$. Thus, since the generators map into the subring $r(DK_1)_0$, then the entire image of the ring $f(\mbbZ[\tilde{X}]) \subseteq r(DK_1)_0$. Thus $\overline{f}(\mbbZ[\tilde{X}]/J) \subseteq r(DK_1)_0$. Therefore, the map $\overline{f}$ is an isomorphism from the subring $\mbbZ[\tilde{X}]/J$ to $r(DK_1)_0$, so by composing with the isomorphism $\pi^*$, we have an isomorphism between $\mbbZ[\tilde{X}]/J$ and $r(K_2)$.
\end{proof}

\subsection{Tensor ideals of the Representation Category of $K_2$}
Recall for any projective $P\in \Rep(K_2)$ and indecomposable $A\in \Rep(K_2)$, $A\otimes P$ is projective. Also, $P(0)$ and $P(1)$ are the projective modules of $K_2$ from Section \ref{sec: Projectives of Kn}. Further, note that by Propositions 2.4 ($V(r')$), 2.25 and 2.33 $(M_n(r',\eta')$), Corollary 2.22 $(\Omega^n(V(r))$, Corollary 2.24 $(\Omega^{-n}(V(r))$ and Corollary 2.10 ($P(r')$) of \cite{chen2014green}, for any indecomposable $A\in \Rep(K_2)$,  $$A\otimes M_n(r,\eta)\simeq c_1 M_k(0,\eta) \oplus c_2 M_k(1,\eta) \oplus c_3 P(0) \oplus c_4(P(1))$$ where $k\leq n$. Further, in $\Rep(DK_1)$ since by the same propositions,  for any indecomposable $A$,  $$A\otimes M_n(r,\eta)\simeq c_1 M_k(0,\eta) \oplus c_2 M_k(1,\eta) \oplus c_3 P(0) \oplus c_4(P(1))\oplus c_5 V(2,0) \oplus c_6 V(2,1)$$ for some $k\leq n$.  \\

\begin{theorem}\label{theorem: additive tensor ideals}
The proper additive tensor ideals of $\Rep(K_2)$ are in bijection with functions $f: \mathbb{C} P^1 \to \mbbZ_+ \cup \{\infty\}$ via the map $\phi$ where $\phi(f)$ is the finite additive completion of $\{M_k(0,\eta), M_k(1,\eta), P(0), P(1) : k\in \mbbZ_+, f(\eta)>k\}$. 
\end{theorem} 

\begin{proof}
Since a retract of a finite direct sum of elements of a set is another finite direct sum of elements a set, $\phi(f)$ is closed under retracts and is additive. By above, $\phi(f)$ is also closed under tensor products. To prove $\phi$ is injective, note that if two functions $f,f'$ are distinct, then without loss of generality, $f(\eta)>f'(\eta)$ for some $\eta\in \mbbC P^1$. Then, letting $n=f'(\eta)+1$, then $M_{n}(0,\eta)\in \phi(f)$ but not in $\phi(f')$ so the map is injective.\\ To prove that $\phi$ surjective, let $\mcT$ be an additive tensor ideal. Since any tensor ideal contains the projective ideal, $P(0),P(1)$ are both in $\mcT$. Suppose the only other indecomposables in $\mcT$ are of the form $M_n(r,\eta)$. If $M_n(r,\eta)\in \mcT$, then \cite[Corollary 2.22]{chen2014green} implies that $M_n(r+1,\eta)\in \mcT$ as well. Also, for each $k<n$, \cite[Proposition 2.33]{chen2014green} implies that $M_k(r,\eta)\in \mcT$ as well. Thus, since additive tensor ideals are closed under retracts and direct sums, they are fully classified by their indecomposable objects. Then, if we define $f(\eta)= \sup\{n: M_n(0,\eta)\in \mcT\}+1$ where the empty supremum is taken to be zero, then since this implies that $\mcT$ and $\phi(f)$ have the same indecomposable objects, then they must be the same tensor ideal. Now, instead suppose $\mcT$ has some other indecomposable object. \\ If either $\Omega^n(V(r))$ or $\Omega^{-n} V(r) \in \mcT$. Then, by \cite[Proposition 2.20]{chen2014green}, $$\Omega^n(V(r)) \otimes \Omega^{-n}(V(r+1))\simeq V(0) \oplus (n^2+n)P(1),$$ so $\mcT$ contains the tensor unit, so it is not proper. If $\mcT$ contains either simple object, then \cite[Proposition 2.1]{chen2014green} implies that the tensor unit is in $\mcT$ as well. Therefore, if $\mcT$ is a proper tensor ideal, the only indecomposables it contains are projective or $M_n(r,\eta)$, so thus each $\mcT$ is the image of some $f$.
\end{proof}

\begin{corollary}
The proper additive tensor ideals of $\Rep(K_2)$ are in bijection with functions $f: \mathbb{C} P^1 \to \mbbZ_+ \cup \{\infty\}$ via the map $\phi$ where $\phi(f)$ is the finite additive completion of $\{M_k(0,\eta), M_k(1,\eta), P(0), P(1), P(2,0), P(2,1) : k\in \mbbZ_+, f(\eta)>k\}$. \end{corollary}

\begin{proof}
Both $V(2,0),V(2,1)$ are in the projective ideal, so they must be contained in any tensor ideal. The injectivity and surjectivity of $\phi$ both immediately follow from Theorem \ref{theorem: additive tensor ideals}.
\end{proof}
\begin{corollary}
    A tensor ideal $\mcT$ of $\Rep(K_2)$ is fully characterized by the set $A_{\mcT}$ where $A_\mcT$ is a subset of the set of all finite subsets of $\mathbb{Z}_+ \times \mathbb{C} P^1 \times \mathbb{Z}_+\times\{0,1\}$ where $T\in A_{\mcT}$ if and only if $\sum\limits_{(n,\eta,k,r)\in T} kM_n(r,\eta)\in \mcT$.
\end{corollary}
\begin{proof}
     We claim this is a complete invariant. Also, note that each $T\in A_{\mcT}$ is finite since each object is a finite direct sum of indecomposables. Suppose $\mcT$ and $\mcT'$ are two distinct proper tensor ideals. Then since $\mcT$ and $\mcT'$ only contain the indecomposables of the form $P(r)$ or $M_n(r,\eta)$, then since for any direct sum of $M_n(r,\eta)$, by tensoring with $\Omega^s V(0)$, we may add arbitrarily many copies of $P(0)$ and $P(1)$. Thus, closure under retracts implies that there must be some element $$\sum\limits_{(n,\eta,k,r)\in A_{\mcT}} kM_n(r,\eta)\in \mcT$$ but not $\mcT'$. However, this in turn implies that $A_{\mcT}\ne A_{\mcT'}$, so this set is a complete invariant. 
\end{proof}

\subsubsection{Quasi-dominated subcategory of $r(K_2)$}

\begin{corollary}
The negligible objects of $\Rep(K_2)$ and $\Rep(DK_1)$  are projectives and $M_n(r,\eta)$.
\end{corollary}

\begin{proof}
Since $K_2$ and $DK_1$ are pivotal Hopf algebras (the element $K$ serves as a pivot in each), the representation categories are pivotal as well. The negligible ideal is a tensor ideal containing any proper tensor ideal, then since $M_n(r,\eta)$ is contained in a proper tensor category, it is negligible. Further, any projective object is negligible. Also, since any tensor ideal with $\Omega^n(V(r))$, $\Omega^{-n}(V(r))$, or $V(r)$ contains the tensor unit, they are not contained in any proper tensor ideal, so they are not negligible. 
\end{proof}

It follows that quasi-dominated subcategory could have uncountably many indecomposables.  This raises doubt about the potential usefulness of quasi-dominated subcategories for the study of non-semisimple modular categories.

\section*{Acknowledgments}
Z.W. is partially supported ARO MURI contract W911NF-20-1-0082. 

\bibliographystyle{abbrv}
\bibliography{references}

\label{VeryLastPage}

\end{document}